\documentclass[12pt]{amsart}
\usepackage{amsmath,amssymb,url}

\theoremstyle{plain}
\newtheorem{thm}{Theorem}[section]
\newtheorem{pro}[thm]{Proposition}
\newtheorem{cor}[thm]{Corollary}
\newtheorem{lem}[thm]{Lemma}
\newtheorem{fact}[thm]{Fact}
\newtheorem{claim}[thm]{Claim}

\DeclareMathOperator{\End}{{\sf End}} 
\DeclareMathOperator{\im}{{\sf im}}  
\DeclareMathOperator{\Lat}{{\sf L}}

\newcommand{\skv}[2]{ \langle #1 \mid  #2 \rangle}

\begin{document}

\title[Representations of $*$-regular rings]{Representations
 of $*$-regular rings and their ortholattices of projections}


\author[C.~Herrmann]{Christian Herrmann}
\author[N.~Niemann]{Niklas Niemann}
\address[C.~Herrmann]{TUD FB4\\Schlo{\ss}gartenstr. 7\\64289 Darmstadt\\Germany}\email{herrmann@mathematik.tu-darmstadt.de}

\begin{abstract}
We show that a subdirectly irreducible $*$-regular ring
admits a representation within some inner product space
provided so does its ortholattice of projections.
\end{abstract}

\maketitle

\section{Introduction}
The motivating examples of $*$-regular rings, due to Murray and
von Neumann, were the $*$-rings
of unbounded operators affiliated with  finite von Neumann algebra factors;
to be subsumed, later, as  $*$-rings of quotients
of finite Rickart $C^*$-algebras. All the latter have been
shown to be $*$-regular and unit-regular (Handelman \cite{hand2}).
Representations of these  as $*$-rings of endomorphisms
of suitable inner product spaces have been
obtained first, in the von Neumann case, by Luca Giudici
(cf. \cite{neu}), in general  in
 joint work with Marina Semenova \cite{awalg}.

The principal right ideals of a $*$-regular ring $R$
form a modular ortholattice $\Lat^\perp(R)$, 
also to be viewed as the ortholattice of projections of $R$.
As observed by Giudici \cite{luca},
any representation of $R$ induces one of $\Lat^\perp(R)$.
Here, a representation of an ortholattice $L$
in an inner product space $V$ means
an embedding $\eta$ of $L$ into the lattice of all linear
subspaces of $V$ such that, for any $u\in L$,
 $\eta(u^\perp)$
is the orthogonal of $\eta(u)$. In his thesis \cite{nik},
the second author established the converse
for subdirectly irreducible $R$ (cf. \cite{nikk}).
This involved a coordinatization of 
representable ortholattices in terms of a variant,
including orthogonality conditions,
of J\'{o}nsson's large partial frames \cite{jon1}.
The purpose of the present note is to give
a short presentation to the result,
relying on the review of Coordinatization Theory 
given in \cite{rep} and the fact
that every variety of $*$-regular rings is generated
by its simple members \cite{simp}.

As general refences we use 
\cite{good, berb, berb2} for regular and $*$-regular rings,
\cite{cd} for lattices, \cite{neu,mae,sko}
for coordinatization, and \cite{gross} for inner product spaces. 
Thanks are due to the referee for a lot of helpful suggestions,
in particular correcting the proof of the crucial Claims
5.4 and 5.6.

\section{Regular rings and vector space representations}

Unless stated otherwise,
rings will be associative,
with  unit $1$ as constant
(constants in the signature have to be preserved
under homomorphisms and in substructures). 
A (von Neumann)  \emph{regular} ring $R$ is  such that for 
each $a \in R$ there is $x \in R$ such that $axa=a$;
equivalently, every right (left) 
principal ideal is generated by an idempotent.

A \emph{representation} of a ring $R$ 
within a vector space $V$ is an
embedding of $R$ into the endomorphism ring $\End(V)$ of $V$. 
It appears to be well known that
every subdirectly irreducible regular 
ring $R$ admits some representation.
Indeed,
each maximal left ideal $M_i$ of $R$
gives rise to a homorphism $\varphi_i:R\to \End(V_i)$,
$\varphi_i(r)(a+M_i)=ra+M_i$;
here  $V_i$ is the (right) vector space 
over the division ring of endomorphisms of the simple left $R$-module  $R/M_i$.
These homomorphisms $\varphi_i$ yield a subdirect representation of $R$
since $\bigcap_i M_i=0$ (for $r\neq 0$ and idempotent $e$ 
with $Rr=Re$ choose $M_i$ such that $1-e\in M_i$ to obtain $r \not\in M_i=\ker \varphi_i$). 
On the other hand, examples of non-representable regular rings 
are obtained as products of matrix rings over fields
of different characteristics.

\medskip
We consider lattices $L$ with  bounds $0,1$ as constants. 
We use $+$ and $\cap$
to denote joins and meets and write $a \oplus b=c$
if $a+b=c$ and $a \cap b=c$.
$L$ is \emph{complemented} if for any $a$ there is $b$
such that $a \oplus b=1$.

The principal right ideals of a regular ring $R$
form a  complemented modular lattice $\Lat(R)$,
a sublattice of the lattice of all right ideals.
A \emph{representation} of a lattice $L$ within a vector space $V$
is an embedding of $L$ into the lattice $\Lat(V)$
of linear subspaces of  $V$.
The following is due to Luca Giudici, proof of (1) in
\cite[Theorem. 4.2.1]{luca}, cf. \cite[Proposition 10.1]{rep}.
\begin{fact}\label{f1} 
If $\iota$ is a representation of the regular ring $R$ 
in the vector space $V$, then
$\eta(aR) =\im \iota(a)$, $a \in R$,
defines  a representation of $\Lat(R)$ in $V$. 
\end{fact}
The purpose of this section
is to relate representations the other way round
making use of coordinatization results
due von Neumann and
J\'{o}nsson, cf. \cite{rep}.
A \emph{coordinatization}
of a lattice $L$ is
an isomorphism
onto $\Lat(R)$, $R$ a regular ring.
Such are based on 
 ''frames'': suitable coordinate systems.
We write $a \sim_c b$ if $a+ b=a\oplus c=b\oplus c$
and $a\sim b$ if $a \sim_c b$ for some $c$.
Recall that, for modular $L$, $a\sim b$ and $a'\leq a$ 
implies $a'\sim b'$ for some $b'\leq b$.
Following J\'{o}nsson~\cite{jon1}
a \emph{large partial $n$-frame} $\Phi$ of $L$
is given 
by elements of $L$:  $a_i=a_{ii}$  $(0\leq i <m)$,
and $a_{0i}$, $0<i<m$, where $m \geq n$,
such that $1=\sum_{i=0}^{m-1} a_i$, $a_0\neq 0$,
 $\sum_{i=0}^{n-1} a_i= \bigoplus_{i=0}^{n-1} a_i$,
  and $a_i\sim_{a_{0i}} b_i$ for some $b_i \leq a_0$ for $0<i<m$.
 Moreover, for $0<i<n$ one requires
$b_i=a_0$.
$\Phi$ is a  {\em skew $n$-$m$-frame} if, in addition, 
$1= \bigoplus_{i=0}^{m-1} a_i$.
 Observe that, given such $\Phi$, $m'\leq m$,
$n'\leq n$, and $n'\leq m'$, the $a_i, a_{0i}$
with $i<m'$ form a
skew $n'$-$m'$-frame  in the interval
$[0,\sum_{i=0}^{m'-1} a_i]$. $\Phi$ is a \emph{skew $n$-frame}
if it is a skew $n$-$m$-frame for some $m$. 
From \cite[Theorem 1.7]{jon1} and \cite[Proposition 6.2]{rep}
one obtains the following
  
\begin{fact}\label{f2}
Every simple complemented modular lattice
of height at least $n$ admits some
large partial $n$-frame.
Every complemented modular lattice  
admitting a large partial $n$-frame also admits
a skew $n$-frame.
\end{fact}
In particular this applies to $\Lat(R)$, $R$ a simple regular ring,
due to the following result of Fred Wehrung \cite[Theorem 4.3]{fredb}.
 
\begin{fact}\label{f3}
For a regular ring $R$,
the lattice of all congruence relations of $\Lat(R)$
is ismorphic to the  lattice of ideals of $R$.
\end{fact}

In presence
of a skew $n$-frame, coordinatization, if possible, is unique
due to the following result of J\'{o}nsson, cf. \cite[Theorem 11.2]{rep}.

\begin{fact}\label{f4}
For regular rings $R,R'$, 
if $\Lat(R)$ admits a skew $n$-frame, $n \geq 3$,
then for any isomorphism $\theta:\Lat(R) \to \Lat(R')$
there is an isomorphism $\iota:R \to R'$ such that
$\theta(aR)= \iota(a)R'$ for all $a \in R$.
\end{fact}

The approach of  \cite{rep} to coordinatization
relied on the following, combining Theorem 7.1 and Corollary 9.2 in \cite{rep}.

\begin{fact}\label{f5}
For any vector space $V$, complemented sublattice $L$ of $\Lat(V)$,
and skew $n$-frame $\Phi$ in $L$, $n\geq 3$,
there is a regular subring $R_0$ of $\End(V)$ and an isomorphism 
$\omega: \Lat(R_0) \to L$
such that $\omega(\varphi R_0) =\im\varphi$ for all $\varphi\in R_0$.
\end{fact}
Now, we are in position to derive a representation of $R$ 
from a representation  of $\Lat(R)$.

\begin{pro}\label{f6}
Given a regular ring $R$, a skew $n$-frame $\Phi$, $n \geq 3$,
in $\Lat(R)$,
a vector space $V$, 
and an embedding $\eta:\Lat(R) \to \Lat(V)$,
there is an embedding $\iota:R \to \End(V)$ such that
$\eta(aR)=\im  \iota(a)$ for all $a \in R$.
\end{pro}

\begin{proof}
Let $L$ denote the sublattice $\eta(\Lat(R))$ of $\Lat(V)$.
With $R_0$ and $\omega$ according to Fact~\ref{f5}
one obtains an isomorphism
$\omega^{-1}\circ \eta:\Lat(R)\to \Lat(R_0)$.
By Fact~\ref{f4}
there is an isomorphism $\iota:R\to R_0$
such that $(\omega^{-1}\circ \eta)(aR)=\iota(a)R_0$
for all $a \in R$. It follows that
$\eta(aR)= \omega(\iota(a)R)= \im \iota(a)$ for all
$a \in R$.
\end{proof}

\section{$*$-regular rings and inner product spaces}
A $*$-\emph{ring} is a ring $R$ endowed with an
involution $r \mapsto r^*$.
Such $R$ is $*$-\emph{regular} if it is regular and  $rr^*=0$ only for $r=0$.
A \emph{projection} is an idempotent $e$ such that $e=e^*$; we write
$e \in P(R)$.
A $*$-ring is
 $*$-regular if and only if 
for any $a \in R$ there is $e \in P(R)$ with
$aR=eR$; such $e$ is unique.
In particular, for  $*$-regular  $R$,
each ideal is closed under the involution. It follows
\begin{fact}\label{f7}
A $*$-regular ring is simple (subdirectly irreducible)
if and only if so is its ring reduct.
\end{fact}
For a $*$-ring $R$ and projection $e \in R$.
the \emph{corner} $eRe$ is the $*$-ring 
consisting of all $eae$, $a \in R$, with unit $e$ 
and operations inherited from $R$, otherwise.
The following is Lemma 2 together with Theorem 3 in
\cite{simp}. 
\begin{fact}\label{f8}
Given a subdirectly irreducible $*$-regular ring $R$ with minimal ideal $I$
and   a projection $e$ in $I$, the $*$-ring $eRe$ is $*$-regular. Moreover,
$R$ is a homomorphic image of a $*$-regular sub-$*$-ring
of an ultraproduct of $*$-rings $eRe$, where $e$ ranges over the set of projections in $I$.
\end{fact}

By an \emph{inner product space} $V$
we will mean a vector space (also denoted by $V$) over a division $*$-ring $F$,
endowed with a sesqui-linear form 
$\skv{.}{.}$
 which
is \emph{anisotropic} ($\skv{v}{v}=0$ only for $v=0$)
 and \emph{orthosymmetric} (
$\skv{v}{w}=0$ if and only if $\skv{w}{v}=0$).
The \emph{orthogonal} of a subset $X$ is the
subspace $X^\perp=\{y \in V\mid \forall x \in X.\,\skv{x}{y}=0\}$.
For subspaces $U,W$ of $V$ we write $U\perp W$ if $W\subseteq U^\perp$;
in this case we write  $U+W=U\oplus^\perp W$. 
A subspace $U$ is \emph{closed} if $U^{\perp\perp}=U$;
equivalently, $V=U\oplus^\perp W$ for some $W$.
Here, $W=U^\perp$ and one has the \emph{orthogonal projection}
$\pi_U$ where $\pi_U(x+y) =x$ for $x \in U$ and $y \in U^\perp$.
Let $\End^*(V)$ denote the $*$-ring
consisting of those endomorphisms $\varphi$
of the vector space $V$ which have an adjoint 
$\varphi^*$ w.r.t. $\skv{.}{.}$.
If $\varphi$ is a projection
in $\End^*(V)$ then $V=\im \varphi \oplus^\perp
\im({\sf id}_V -\varphi)$. It follows

\begin{fact}\label{f9}
An endomorphism $\varphi$ of $V$ is a projection in $\End^*(V)$ if and
only if $\varphi=\pi_U$, where $U=\im \varphi$.
\end{fact}

A \emph{representation} of a $*$-ring $R$ within
$V$ is an embedding of $R$ into $\End^*(V)$.
Of course, any representation $\iota$ of a  $*$-ring $R$
within  $V$ gives rise to   representations
of   corners $eRe$ within $\im \iota(e)$.

Inner product spaces will be considered as 
$2$-sorted structures with sorts $V$ and $F$.
In particular, the class of inner product spaces is
closed under formation of ultraproducts. 
In this setting,
representations of 
$*$-rings $R$ can be viewed as $3$-sorted structures
(with third sort $R$),
again forming a class closed under ultraproducts 
\cite[Proposition 13]{awalg}. On the other hand,
a representation of $R$ in $V$ gives rise to 
representations of homomorphic images of $R$ in
closed subspaces of certain ultrapowers of $V$
\cite[Proposition 25]{awalg}. It follows

\begin{fact}\label{f10}
In the context of Fact~\ref{f8}, if  each
$*$-ring $eRe$  admits a representation within some inner product space $V_e$
for each projection $e \in I$,
 then the $*$-ring $R$ admits a representation within  a closed subspace of an ultraproduct of spaces  $V_e$, where $e$ ranges over the set of projections in $I$.
\end{fact}

\section{Modular ortholattices}
An \emph{ortholattice} is a lattice $L$ together with an
order reversing involution $a \mapsto a^\perp$
such that $1=a\oplus a^\perp$.
Elements $a,b$ are \emph{orthogonal} to each other, $a \perp b$,
if 
$b\leq a^\perp$; this then implies $a \cap b=0$
and we  write $c= a\oplus^\perp b$ if $c=a+b$.
If $L$ is modular and  $u \in L$,
then  the section $[0,u]$ is again an ortholattice
under $a \mapsto u \cap a^\perp$;
that is, $a,b \leq u$ are orthogonal in $[0,u]$ if and only
if they are so in $L$.
Also,
if $L$ is modular and $a \leq b$ then each of the quotients
$b/a$,
$(b \cap a^\perp)/0$, and $a^\perp/b^\perp$  generate the
same lattice congruence. It follows

\begin{fact}\label{f11}
In a modular ortholattice, any lattice congruence 
is also a congruence w.r.t. the operation $a \mapsto a^\perp$.
\end{fact}

The notion of skew frame can be adapted to the ortholattice setting
requiring the $a_i$ to be pairwise orthogonal, see Niemann
\cite{nik}. A weaker version will suffice, here.
We write  $a \sim ^\perp b$ if $a \perp b$ and 
$a \sim b$. An \emph{orthogonal semiframe} in an
ortholattice $L$ consists of elements $a_0, \ldots ,a_{k-1}$
such that $1=\bigoplus_{i=0}^{k-1} a_i$ and for 
each $a_i$ there is $b_i\sim^\perp a_i$.
 
\begin{lem}\label{l1}
Every modular ortholattice $L$ admitting some
skew   $2$-$m$-frame  also admits an orthogonal semiframe.
In particular, any
simple $L$ of height at least $2$ admits an orthogonal semiframe.
\end{lem} 
\begin{proof}
We first observe that the following hold in any modular ortholattice.
\begin{itemize}
\item[(1)] If  $v \oplus b=1$ and $v^\perp  \cap b=0$
then $v^\perp \sim^\perp v'$ for some $v' \leq v$.
\item[(2)]  Assume $u\oplus a=1$ 
and $a \sim^\perp a'$ for some $a'\leq u$.
Then there are $d,f$ such that $1=u\oplus^\perp d \oplus^\perp f$,
$d \sim^\perp d'$ and $e \sim^\perp e'$ for
some $d' \leq u$ and $e'\leq u+d$.
\end{itemize}
(1) follows from $v^\perp \sim_b v\cap (v^\perp +b)$. 
To prove (2), put $d:=a \cap u^\perp$ and $v:=u+d$. 
Then $v=u\oplus^\perp d$ and $d\sim^\perp d'$ for
some $d' \leq a' \leq u$. Moreover, $ a \cap v^\perp = a \cap d^\perp \cap
u^\perp = d \cap d^\perp =0$.
Now, put $b:= a \cap d^\perp$, the orthocomplement of
$d$ in the ortholattice $[0,a]$; thus, $b\oplus d=a$
and $v\oplus b=1$. 
On the other hand, from $b \leq a$ it follows $b \cap v^\perp=0$. 
Now, $1=u\oplus d \oplus v^\perp$ and  (2) follows applying (1). 

Finally, observe 
that (2)  deals with the case $m=2$ as well with the
inductive step from $m-1$ to $m$.
The second claim follows from Facts~\ref{f2} and \ref{f11}
\end{proof} 
A \emph{representation} of an ortholattice $L$ in an
inner product space $V$ 
is an embedding $\eta$ of
the lattice  $L$ into $\Lat(V)$
such that $\eta(a^\perp)=\eta(a)^\perp$ for all $a \in L$.
In particular, $\eta(L)$ is a modular sub-ortholattice
of the (in general, non-modular)  lattice of all closed 
subspaces of $V$.

\begin{fact}\label{f12}
Given a $*$-regular ring $R$,
the lattice $\Lat(R)$ expands to an
ortholattice $\Lat^\perp(R)$ 
defining $(aR)^\perp =(1-e)R$
where $e\in P(R)$ such that $aR=eR$.
In particular, for $e,f \in P(R)$ one has 
$eR\subseteq  fR$ if and only if $fe=e$.
\end{fact}
For $e,f \in P(R)$, we  write $e \perp f$ if $eR \perp fR$;
that is, $fe=0=ef$.
Now, in view of Fact~\ref{f9},  Fact~\ref{f1} transfers as follows.
\begin{fact}\label{f13} 
If $\iota$ is a representation of the $*$-regular ring $R$ 
in the inner product space $V$ then 
$\eta(aR) =\im \iota(a)$, 
defines  a representation of the ortholattice  $\Lat^\perp(R)$ in $V$.
\end{fact}
In the presence of orthogonal semiframes, we will
relate such representations the other way round.

\section{Main Lemma}

\begin{lem}\label{l2}
Given a $*$-regular ring $R$, an orthogonal semiframe $\Phi$
in $\Lat^\perp(R)$,
an inner product  space $V$, 
and representation $\iota$ of the ring $R$in the vector space $V$, 
then $\iota$ is a representation of the $*$-ring $R$ 
within $V$, provided that
 $\eta:\Lat^\perp(R) \to \Lat(V)$, 
$\eta(aR):=\im  \iota(a)$, $a \in R$,
defines  an ortholattice representation in the inner product space $V$. 
\end{lem}
Recall that $\eta(aR)=\eta(bR)$ if $aR=bR$
(and we may write $\eta(a):= \eta(aR)$)
and that $\eta$ is a lattice representation in view of Fact~\ref{f1}.
Thus, the point is
to show  $\iota(a^*)=\iota(a)^*$ for all $a \in R$
using  the fact that 
$\eta$  preserves orthogonality.
For the remainder of this section we assume the
hypotheses of the Lemma.

\begin{claim}\label{c1}
Consider closed subspaces $U,W$ of $V$
such that  $U\perp W$ 
and $\varphi,\psi \in \End(V)$ such that $\varphi=\pi_W \varphi \pi_U$
and $\psi=\pi_U\psi \pi_W$.
Then $\psi=\varphi^*$ if and only if 
$\im (\pi_U -\varphi)\perp \im (\pi_W+\psi)$.
\end{claim}
\begin{proof}
This follows immediately since for all  $v,w \in V$ one has
\[\skv{(\pi_U-\varphi)(v)}{(\pi_W+ \psi)(w)}
=\skv{\pi_U(v)}{\psi(w)} - \skv{\varphi(v)}{\pi_W(w)} \]
\end{proof}

\begin{claim}\label{c2}
If $e \perp f$ in $P(R)$ and $a \in fRe$
then  $\iota(a^*)=\iota(a)^*$.
\end{claim}
\begin{proof}
Put  $b=a^*$. Then $b \in eRf$ and
$(e-a)^*(f+b)=0$, that  is $(e-a)R \perp (f+b)R$. 
It follows $\eta(e-a) \perp \eta(f+b)$. Now,
$\eta(e-a)= \im \iota(e-a)= \im(\iota(e) - \iota(a))$
and $\eta(f+b)= \im(\iota(f) + \iota(b))$
and Claim~\ref{c1} applies with $\varphi=\iota(a)$,
$U=\im \iota(e)$, $\psi=\iota(b)$,
$W=\im \iota(f)$. 
\end{proof}

\begin{claim}\label{c3}
If $eR \sim fR$ in $\Lat^\perp(R)$
for  idempotents
$e,f \in R$  then there is $c \in fRe$ such that
$cx=cy$ implies $x=y$ for all $x,y \in eRe$.
\end{claim}
\begin{proof}
Assume $eR \sim_{gR} fR$;
then $\omega(x)=y \Leftrightarrow x-y\in gR$
defines an  isomorphism
$\omega:eR \to fR$ of right $R$-modules.
Put $c= \omega(e)e$ and observe that
$\omega(x)=\beta(ex)=\beta(e)x =  cx$
for all $x \in eRe$.  Thus, assuming $cx=cy$ for given
$x,y \in eRe$ it follows
$\omega(x)=cx=cy=
\omega(y)$ whence $x=y$.  
\end{proof}

\begin{claim}\label{c4}
Consider closed subspaces $U\perp W$ of $V$ and
$\varepsilon \in \End^*(V)$ such that
$\varepsilon\circ \xi = \varepsilon\circ \chi$ 
implies $\xi= \chi$ for all
$\xi,\chi \in \pi_U\circ\End^*(V)\circ\pi_U$. 
 Then $\varphi^*=\psi$ provided that
$\varphi,\psi \in \pi_U \circ\End^*(V)\circ\pi_U$ and
 $(\varepsilon\circ  \varphi)^*=
\psi \circ \varepsilon^*$.
\end{claim}
\begin{proof}
From $\varphi^* \circ \varepsilon^*= (\varepsilon \circ  \varphi)^*
=\psi \circ \varepsilon^*$ 
it follows $\varepsilon \circ\varphi= \varepsilon \circ \psi^*$,
whence  $\varphi= \psi^*$ and $\varphi^* =\psi$.
\end{proof}
\begin{claim}\label{c5}
Given $e,f$ as in Claim~\ref{c3} such that $e \perp f$ one has 
 $\iota(a^*) =\iota(a)^*$ for all $a \in eRe$.
\end{claim}
\begin{proof} 
We put $b=a^*$   and have to show $\iota(b) =\iota(a)^*$.
Choose $c$ according to Claim~\ref{c3}.
By Claim~\ref{c2} one has $\iota(c^*)=\iota(c)^*$
and $\iota((ca)^*)= (\iota(ca))^*$ since $c,ca \in fRe$.
It follows $\iota(b)\iota(c)^* =
\iota(b) \iota(c^*)= \iota(bc^*) =\iota((ca)^*)= (\iota(ca))^* 
=(\iota(c) \iota(a))^*  $ whence
$\iota(b) =\iota(a)^*$
applying Claim~\ref{c4} with
$U=\im e$, $W=\im f$,  $\varphi=\iota(a)$,
$\psi= \iota(b)$, and $\varepsilon=\iota(c)$.
\end{proof}

\begin{proof} of the Lemma. We fix an orthogonal semiframe $\Phi$
of $\Lat^\perp(R)$, that
is pairwise orthogonal projections $e_i$, $0\leq i <k$,
such that $\bigoplus_{i=0}^{k-1}e_iR=R$ and for each $i<k$
there are 
$f_i, g_i \in P(R)$ with   
$e_iR \sim^\perp  f_iR$.
By Claims~\ref{c2} and \ref{c5} one has $\iota(a^*)= \iota(a)^*$
for all $a \in e_jRe_i$, $i,j<k$. 

Now, $e_ie_j=0$ for $i\neq j$ since $e_i\perp e_j$.
Thus $e=\sum_{i0}^{k-1}e_i$ is a projection and $eR=R$ whence
$e=1$ by uniqueness. It follows, for each $a \in R$,
that $a=\sum_{i,j=0}^{k-1} e_jae_i$
and
$a^* = \sum_{i,j=0}^{k-1} e_ja^*e_i$. Thus
$\iota(a)^*
= (\sum_{i,j=0}^{k-1} \iota(e_ja e_i))^*
= \sum_{i,j=0}^{k-1} (\iota(e_ja e_i))^*
=\sum_{i,j=0}^{k-1} \iota( (e_jae_i)^*)
=\iota(\sum_{i,j=0}^{k-1} (e_jae_i)^*)
=\iota(\sum_{i,j=0}^{k-1} e_ia^*e_j)
=\iota(a^*) $.
\end{proof}

\section{Results}

Facts~\ref{f7}, \ref{f3}, and \ref{f2} yield the following.
\begin{fact}
A simple $*$-regular ring admits a large partial $n$-frame
if $\Lat(R)$ is of height at least $n$.
\end{fact}

\begin{thm}\label{t1}
Given a $*$-regular ring $R$
such that its ortholattice $\Lat^\perp(R)$
of projections admits a large partial $n$-frame, $n \geq 3$,
and a representation $\eta$ within some inner 
product space $V$. Then there is a representation $\iota$ of
the $*$-ring $R$
within $V$ such that $\eta(a)=\im \iota(a)$ for all $a \in R$.
\end{thm}
\begin{proof}
By Fact~\ref{f2}  one has a skew $n$-frame, $n\geq 3$ 
and so Proposition~\ref{f6} provides the ring
embedding $\iota: R\to \End(V)$
such that $\eta(a)=\im \iota(a)$. Now, by Lemma~\ref{l1}
there is an orthogonal semiframe and Lemma~\ref{l2}
shows that $\iota$ is a representation of the $*$-ring $R$.
\end{proof}

\begin{cor}\label{t2}
Consider  a subdirectly irreducible $*$-regular ring $R$ such that
$\Lat^\perp(R)$ is of height at least $3$ and has a representation in the inner product space $V$.
Then the $*$-ring $R$ has a representation
within a closed subspace of some ultrapower of $V$.
\end{cor}
\begin{proof}
Let $P$ denote the set of projections $e$ in the minimal ideal $I$
of $R$ such that $\Lat^\perp(eRe)$ has height at least $3$.
 Observe that for any projection
$f \in I$, $fRf$ embeds into $eRe$
for some $e \in P$.
Also, for $e\in P$,
$\Lat^\perp(eRe)$ is
 a section of $\Lat^\perp(R)$
and  any  representation $\eta$ of
$\Lat^\perp(R)$    in some inner product space  $V$
restricts to a representation of
$\Lat^\perp(eRe)$ in a  closed subspace $V_e$ of $V$.
By Lemma~\ref{l2} one obtains a representation
of the $*$-ring  $eRe$ within $V_e$ for each $e \in P$.
By Fact~\ref{f10} this gives rise to a representation
of $R$ in an ultraproduct of the $V_e$,
that is a closed subspace of an ultrapower of $V$.
\end{proof}

Let  $\mathcal{MOL}$ and  $\mathcal{MOL}_{art}$ denote
the ortholattice varieties generated by all respectively
all finite height modular ortholattices.

\begin{cor}\label{t3}
A  $*$-regular ring $R$
is a subdirect product of representables if and only if $\Lat^\perp(R)\in \mathcal{MOL}_{art}$.
\end{cor}
\begin{proof}
Consider a homomorphism $\iota_k$ of $R$ onto $S_k$.
Then $S_k$ is also $*$-regular and 
$\iota_k$ induces a homomorphism $\eta_k$ of $\Lat^\perp(R)$
onto $\Lat^\perp(S_k)$ given by $\eta_k(eR)= \iota_k(e)S_k$, see
Proposition 5.4(iv) \cite{linrep}.
Moreover,  the $\iota_k$ yield a subdirect decomposition
if and only if so do the
$\eta_k$.

Thus, it suffices to consider  subdirectly irreducible $R$;
that is, subdirectly irreducible $\Lat^\perp(R)$.
If $R$ is representable then, by Fact~\ref{f3},
$\Lat^\perp(R)$ is representable, too, and
so in the variety generated by subspace ortholattices
of finite dimensional inner product spaces by \cite[Theorem 10.1]{linrep}, 
whence in $\mathcal{MOL}_{art}$.
For the converse, we may assume that $\Lat^\perp(R)$ 
is of height at least $4$, since otherwise $R$
is simple artinian whence representable. 
Now any finite height modular ortholattice
is a direct product of simple ones and, by  J\'{o}nsson's Lemma, $\Lat^\perp(R)$ is
is a homomorphic image of a sub-ortholattice
of an ultraproduct of such $L_i$.
Since $\Lat^\perp(R)$ contains a 5-element chain,
the ultraproduct may be restricted to be formed from the
$L_i$ of height at least $4$. Such are representable
whence, by Lemma 8.3 and Corollary 8.6 in \cite{linrep}, so is $\Lat^\perp(R)$.
The claim follows by Corollary~\ref{t2}.    
\end{proof}

\begin{cor}
 $\mathcal{MOL}=\mathcal{MOL}_{art}$
if and only if  
 every subdirectly irreducible $*$-regular ring is
representable. 
\end{cor}
\begin{proof}
Assume that every subdirectly irreducible $*$-regular
ring is representable.
According to \cite{hr} it suffices  to show
$L\in \mathcal{MOL}_{art}$ for each simple $L\in \mathcal{MOL}$.
Of course, we have to consider $L$ of infinite height, only.
Such $L$ is coordinatizable, that is $L\cong \Lat^\perp(R)$ for
some $*$-regular ring $R$ (cf.
\cite[Section 4.3]{flo}
respectively  Fact~\ref{f2} and Theorems 11.2  and  13.2 in \cite{rep}).  Thus, by Fact~\ref{f3}, $L$ is representable,
whence in $\mathcal{MOL}_{art}$  cf. \cite[Theorem 10.1]{linrep}.
The other direction is immediate by Corollary~\ref{t3}.
\end{proof}

In case of $*$-regular rings $R$ without unit,
$R$ is the directed union of the $eRe$,
 $e$ a projection in $R$, and 
$\Lat^\perp(R)$  the directed union of the
$\Lat^\perp(eRe)$. The latter is a modular \emph{sectional ortholattice} $L$,
a modular lattice with $0$ and a
binary  operation $(a,u) \mapsto a^{\perp_u}$
such that $a  \mapsto a^{\perp_u}$ is an orthocomplementation 
on $[0,u]$ and  $a^{\perp_v}= a^{\perp_u} \cap v$ 
if $v \leq u$. A  \emph{representation} of such $L$
is given by an inner product space $V$ and an embedding $\eta$
of the lattice $L$ with $0$ into $\Lat(V)$
such that, for each $u\in L$,  $\eta(u)$ is closed in $V$ and
the restriction of $\eta$ is a representation
of the ortholattice  $[0,u]$ within $\eta(u)$. 
\begin{cor}
Corollaries~\ref{t2} and \ref{t3} hold for $*$-regular rings without unit,
analogously. 
\end{cor}

\end{document}